\documentclass[11pt]{amsart}%Roum
\usepackage{amsmath,amsfonts,amsthm,amssymb}
\usepackage{hyperref}
\usepackage{graphicx}
\usepackage{color}

\theoremstyle{plain}
\newtheorem{theorem}{THEOREM}[section]
\newtheorem{lemma}[theorem]{LEMMA}
\newtheorem{proposition}[theorem]{PROPOSITION}
\newtheorem{corollary}[theorem]{COROLLARY}

\theoremstyle{remark}
\newtheorem*{note}{Note}
\newtheorem*{definition}{Definition}

\DeclareMathOperator{\supp}{supp}
\newcommand{\R}{\mathbb{R}}

\begin{document}

\title[Binormal measures]{Binormal measures}
\author{Emmanuel P. Smyrnelis}

%\address{Department of Mathematics\\ University of Athens\\ Panepistemiopolis\\ 15784 Athens\\ Greece}
%\email{\href{mailto:smyrnel@math.uoa.gr}{\texttt{smyrnel@math.uoa.gr}}}

\author{Panayotis Smyrnelis}
%\address{Basque Center for Applied Mathematics, Alameda de Mazarredo 14, 48009 Bilbao, Spain}
%\email{psmyrnelis@bcamath.org}


%\date{\today}

\begin{abstract}
Our starting point is the  measure $\epsilon_x-\alpha_x\rho_x^{\omega_1}+\beta_x\rho_x^{\omega_2}$, where $\rho_x^{\omega_i}$ is the harmonic measure relative to $x \in \omega_1 \subset \overline{\omega}_1 \subset \omega_2$ and $\omega_i$ are concentric balls of $\R^n$; $\alpha_x$, $\beta_x$ are functions depending on $x$ and on the radii of $\omega_i$, $(i=1,2)$. Generalizing the above measure, we introduce and study the binormal measures as well as their relation to biharmonic functions.


\end{abstract}

\maketitle
\section{\textbf{INTRODUCTION}}\label{sec1}

The characteristic mean value property of harmonic (resp. parabolic) functions involves the  measures $\lambda=\epsilon_x-\rho^{\omega}_x$, where $\epsilon_x$ is the Dirac measure at $x \in \omega$ and $\rho^{\omega}_x$ is the harmonic (resp. parabolic) measure relative to $\omega$ and $x \in \omega$, supported by the sphere $\partial \omega$ (resp. by the level surface $\partial \omega$ of the heat kernel). The adjoint potential of these measures is equal to zero on $\complement \overline{\omega}$ (the complement of $\overline{\omega}$), or equivalently, their swept measures satisfy $\lambda^{\complement \overline{\omega}}=0$. %More general measures have the same properties (see \cite{choquet-deny},\cite{smyrnel1},\cite{smyrnel2}).

In 1944, G. Choquet and J. Deny, generalized the measure $\epsilon_x-\rho_x^\omega$, and  introduced the normal distribution. Moreover, they proved some characteristic properties of solutions of the equations $\Delta u=0$, and $\Delta^p u=0$ in $\R^n$. Next, in 1967, de La Pradelle following an idea of J. Deny \cite{deny}, extended the notion of normal measure to the frame of Brelot's theory \cite{brelot1}. Finally, in 1971, E. Smyrnelis using the extended notion of normal measure, proved several characteristic properties of normal measures and harmonic functions in Brelot spaces, applicable to solutions of $Lu=0$, where $L$ is a second order linear elliptic operator in $\R^n$.

On the other hand, biharmonic functions (that is, solutions of $\Delta^2 u=0$) satisfy a mean value property which involves the measures $\epsilon_x-\alpha_x\rho_x^{\omega_1}+\beta_x\rho_x^{\omega_2}$, where $\alpha_x$, $\beta_x$ are functions of $x \in \omega_1 \subset \overline{\omega}_1 \subset \omega_2$ and of the radii $R_1$, $R_2$ of the concentric spheres $\partial \omega_1$, $\partial \omega_2$ (cf. \cite{smyrnel5}). The scope of this article is to generalize this property and study some related issues, for the solutions of the equation $(L_2L_1)h=0$, where $L_i$ $(i=1,2)$ is a second order linear elliptic differential operator. The idea is to work in an elliptic biharmonic space, and use special general measures, applicable in particular, to the above equation; note that to this elliptic biharmonic space, we associate a $1$-harmonic and a $2$-harmonic space that in the applications correspond respectively to the solutions of the equations $L_1 h=0$, and $L_2 u=0$.

To this end, we first introduce in Section \ref{secc2} the binormal pair of measures $\Theta=(\lambda,\mu)$, supported by the compact set $K$, as the pair such that the swept measures on $\complement K$  of  $\Lambda:=(\lambda,0)$ and $ M:=(0,\mu)$ vanish. Since $\Theta=\Lambda+M$,  it follows that $\Theta^{\complement K}=\Lambda^{\complement K}+M^{\complement K}$ or $(\lambda,\mu)^{\complement K}=(\lambda,0)^{\complement K}+(0,\mu)^{\complement K}$ (cf. \cite{smyrnel3}).

The pair $(\lambda,0)$ is called a pure biharmonic pair if $(\lambda,0)^{\complement K}=(0,0)$. This is equivalent to the condition: `the pure adjoint potential pair vanishes on $\complement K$'.

The pair $(0,\mu)$ is called $2$-normal if $(0,\mu)^{\complement K}=(0,0)$ or, equivalently, if the $2$-adjoint potential vanishes on $\complement K$,
(see\cite{smyrnel2}).

Several examples of the aforementioned pairs of measures are given in Section \ref{sec:sec3}.
% and we show that the set of the pure binormal pairs is contained in the set of the $1$-normal pairs.

%We prove characteristic mean value properties of biharmonic pairs of functions.

%We also obtain interesting properties of binormal pairs of measures as, for example, the one connecting the thinness of $\complement K$ with the existence of a non trivial pure binormal pair of measures.

In Section \ref{secc4}, we prove characteristic mean value properties of
biharmonic pairs in relation to biharmonic pairs of measures.

Section \ref{secc5} is devoted to the study of properties of binormal pairs of measures.
Furthermore, we show that the linear combinations of the pairs $(\epsilon_x - \mu^{\complement K}_x, \nu^{\complement K}_x)$ are dense, for the vague topology, in the space of the pure binormal pairs of measures, where $(\mu^{\complement K}_x, \nu^{\complement K}_x) = (\epsilon_x, 0)^{\complement K}$. Analogous results hold for the  measures $\epsilon_x - \mu^{\complement K}_x$ (resp. $\epsilon_x - \lambda^{\complement K}_x$) in the space of $1$-normal (resp. $2$-normal) measures, where $\mu^{\complement K}_x$ is the swept nonnegative  measure of $\epsilon_x$ in the $1$-harmonic space (resp. $\lambda^{\complement K}_x$ is the swept nonnegative  measure of $\epsilon_x$ in the $2$-harmonic space). Finally, we examine the relation between binormal and normal measures.

\begin{note}In this work, we use the term `measure'  for `signed measure'.\end{note}

\section{\textbf{REMINDERS, DEFINITIONS AND PRELIMINARY RESULTS}}\label{secc2}

We shall work in a locally compact, connected space $\Omega$ with a countable basis. We denote by $\mathcal{U}$ (resp. $\mathcal{U}_c$) the set of all nonempty open sets (resp. the set of all nonempty relatively compact open sets), in $\Omega$.

Let $\mathcal{H}$ be a map that associates to each $U \in \mathcal{U}$ a linear subspace of  $C(U) \times C(U)$ which is composed of compatible pairs $(u_1, u_2)$ in the sense that if $u_1=0$ on an open set then $u_2=0$ there. The pairs of $\mathcal{H}(U)$ are called biharmonic on $U$.

On the other hand, a set $\omega \in \mathcal{U}_c$ with $\partial\omega\neq\varnothing$ is called $\mathcal{H}$-regular if
\begin{itemize}
\item[i)] The Riquier boundary value problem has only one solution $(H_{1}^{\omega,f}, H_{2}^{\omega,f})$ associated to the pair $f=(f_{1},f_{2})\in C(\partial\omega)\times C(\partial\omega)$.
\item[ii)] $f_{j}\geq 0$ $(j=1,2)$ implies $H_{1}^{\omega,f}\geq 0$, while $f_{2}\geq 0$ implies $H_{2}^{\omega,f}\geq 0$. Hence, there exists for every $x \in \omega$, a unique system $(\lambda_x^\omega,\mu_x^\omega,\nu_x^\omega)$ of Radon nonnegative measures on $\partial\omega$ such that $ H_{1}^{\omega,f}(x)=\int{f}_1\,d\mu_x^\omega
+\int\ {f}_{2}\,d\nu_x^\omega$, $H_{2}^{\omega,f}(x)=\int{f}_2\,d\lambda_x^\omega$.
\end{itemize}

A pair of functions $(v_1, v_2)$ defined on $U \in \mathcal{U}$ such that $v_{j}: U\longrightarrow \left(-\infty,+\infty\right]$ is called hyperharmonic if
\begin{itemize}
\item $v_j$ is lower semi-continuous,
\item and $ v_{1}(x)\geq\int{v}_1\,d\mu_x^\omega+\int\ {v}_{2}\,d\nu_x^\omega$,\:$v_{2}(x)\ge\int{v}_2\,d\lambda_x^\omega$, holds for every regular set $\omega \subset\overline{\omega}\subset U$ and every $x\in\omega$.
\end{itemize}
If the function $v_1$ is finite on a dense subset of $U$, then the hyperharmonic pair $(v_1, v_2)$ is called superharmonic on $U$. Finally, a  nonnegative  superharmonic pair $p=(p_{1},p_{2})$ will be called potential pair (on $U$), if $(h_{1},h_{2})= (0,0)$ is the only biharmonic pair satisfying $0\leq h_{j}\leq p_{j}\:(j=1,2)$.

The space $(\Omega, \mathcal{H})$ with the axioms  I, II, III, IV  is called biharmonic (see \cite{smyrnel3}). A biharmonic space is called elliptic if, for every $x\in\Omega$ and every regular set $\omega\owns x$, $\supp(\lambda_{x}^{\omega})=\supp(\mu_{x}^{\omega})=\supp(\nu_{x}^{\omega})= \partial\omega$; it will be called strong if there exists a strictly positive potential pair on $\Omega$. To a biharmonic space, we associate the underlying harmonic spaces $(\Omega,\mathcal{H}_{1})$,\:$(\Omega,\mathcal{H}_{2})$, which correspond respectively to the solutions of $L_{1}u_{1}=0$,\:$L_{2}u_{2}= 0$ in the classical case. We use respectively the prefixes $1$ or $2$, to refer to the aforementioned harmonic spaces.

Next, the hyperharmonic (resp. superharmonic; potential) pair $(v_1, v_2)$ is said pure, if given a nonnegative $2$-hyperharmonic function $v_2$ on $U$,\:  $v_{1}$ is the smallest nonnegative function  such that  $(v_{1}, v_{2})$ is a nonnegative hyperharmonic (resp. superharmonic; potential) pair on $U$. The $j$-harmonic (resp. biharmonic) support of a $j$-hyperharmonic function (resp. hyperharmonic pair) is defined as the smallest closed set such that the function (resp. the pair) is $j$-harmonic (resp. biharmonic) in its complement (j=1,2). We call Green's pair a pure potential pair with punctual biharmonic support. We also recall that if $\varphi$ is a numerical function on an open set $U$, the function $\widehat{\varphi}$
is defined as follows:$$\widehat{\varphi}(x)= \underset{\underset{y\in U}{y\to x}}{\lim}\inf \varphi(y).$$

In \cite{smyrnel8}, we define and study the adjoint biharmonic spaces corresponding to the adjoint equation $(L_2L_1)^* h= 0$, or  equivalently to the system $$L_{2}^{\ast}h_{2}=-h_{1},\ L_{1}^{\ast}h_{1}=0.$$ The asterisk symbol is used in the sequel to refer to adjoint spaces.

Our framework will be an elliptic strong biharmonic connected space. We assume the proportionality of $i$-Green's potentials and $i$-adjoint Green's potentials, and also the existence of a topological basis of completely determining domains for the associated $i$-harmonic spaces $(i=1,2)$. For the notions and notation not explained in this work, we refer to \cite{smyrnel3,herve}.

\begin{definition}\label{def1}
 Let $\lambda, \mu$ be Radon measures supported by a compact set $K\subset\Omega$ and let $\lambda = \lambda_1 - \lambda_2$, $\mu = \mu_1 - \mu_2$, with $\lambda_j\geq 0$, $\mu_j\geq 0$, $(j=1,2)$.
\begin{itemize}
\item The pair $(\lambda, \mu)$ is called binormal for $K$ if both $(\lambda, 0)^{\complement K} = (0,0)$ and $(0,\mu)^{\complement K} = (0,0)$.
\item The pair $(\lambda, 0)$ is called pure binormal for $K$ if $(\lambda, 0)^{\complement K} = (0, 0)$.
\item The pair $(0, \mu)$ is called  $2$-normal for $K$ if $(0, \mu)^{\complement K} = (0, 0)$.
\end{itemize}
\end{definition}
Let us consider the open subset $\omega \subset \Omega$, the points $x, y \in \Omega$, the Green's pair $(w_y, p^{2}_{y})$ of biharmonic support $\{y\}$ and the adjoint Green's pair $(w^{*}_{x}, p^{1*}_{x})$ of support $\{x\}$ (cf. \cite{smyrnel7,smyrnel8}). We denote by  $(W^{\omega}_{y}, P^{2,\omega}_{y})$  the swept pair on $\omega$ of the former pair, and by $(W^{*,\omega}_{x}, P^{1*,\omega}_{x})$ the swept pair on $\omega$ of the latter pair. We also consider the adjoint pure potential pair $p^{*}_{\nu} = (w^{*}_{\nu}, p^{1*}_{\nu})$ with associated nonnegative measure $\nu$, where $w^{*}_{\nu}(u)
= \int w^{*}_{x}(u)d\nu(x)$,\: $p^{1*}_{\nu}(u) = \int p^{1*}_{x}(u)d\nu(x)$, and $(W^{*,\omega}_{\nu}, P^{1*,\omega}_{\nu})$ its swept pair corresponding to the open set $\omega$.
\begin{lemma}\label{lemma1} We assert that $W^{*,\omega}_{\nu}(y) = \int W^{*,\omega}_{x}(y)d\nu(x) = \int W^{\omega}_{y}(x)d\nu(x)$.
\end{lemma}
\begin{proof} If  $(\alpha^{\omega}_{y}, \beta^{\omega}_{y})$  is the adjoint swept pair of $(\epsilon_y, 0)$ on $\omega$, then
\begin{align*}
W^{*,\omega}_{\nu}(y)&= \int w^{*}_{\nu}(u)d\alpha^{\omega}_{y}(u) + \int p^{1*}_{\nu}(u)d\beta^{\omega}_{y}(u) \\
&= \int \Big(\int w^{*}_{x}(u)d\nu(x)\Big) d\alpha^{\omega}_{y}(u)+\int \Big(\int p^{1*}_{x}(u)d\nu(x)\Big)d\beta^{\omega}_{y}(u) \\
&=\int \Big(\int w^{*}_{x}(u)d\alpha^{\omega}_{y}(u)\Big)d\nu(x) + \int \Big(\int p^{1*}_{x}(u)d\beta^{\omega}_{y}(u)\Big)d\nu(x).
\end{align*}
Since $W^{*,\omega}_{x}(y)=\int w^{*}_{x}(u)d\alpha^{\omega}_{y}(u)+\int p^{1*}_{x}(u)d\beta^{\omega}_{y}(u)$, using \cite[Lemma 4]{smyrnel8} and a  remark after the proof of  \cite[Proposition 4.2.]{smyrnel8}, we obtain  $W^{\omega}_{y}(x)=W^{*,\omega}_{x}(y)$, which completes the proof.
\end{proof}

\begin{theorem}\label{th1} Let $(\lambda, 0)$ be a  pair of  measures supported by the compact set $K$. Then, the following properties are equivalent:\begin{itemize}\item[(i)] $(\lambda, 0)$ is pure binormal relative to $K$.\item[(ii)] The adjoint pure potential pair $(w^{*}_{\lambda}, p^{1*}_{\lambda})$ vanishes on  $\complement K$.\end{itemize}
\end{theorem}
\begin{proof}  First, we notice  that as the pair $(w^{*}_\lambda,  p^{1*}_\lambda)$ is adjoint biharmonic on $\complement K$, and therefore compatible, if $w^{*}_{\lambda_1}$ = $w^{*}_{\lambda_2}$ on $\complement K$, then $p^{1*}_{\lambda_1}$ = $p^{1*}_{\lambda_2}$ there. (In other words,  if $(\lambda, 0)$ is pure binormal, then $\lambda$  is $1$-normal.)\\  $(ii) \Rightarrow (i)$.
 $w^{*}_{\lambda_1} = w^{*}_{\lambda_2}$  on $\complement K$   implies that the respective reduced functions satisfy  $W^{*\complement K}_{\lambda_1}$ = $W^{*\complement K}_{\lambda_2}$ in $\Omega$  and, by Lemma 2.2., $\int W^{\complement K}_y (x)d\lambda_1(x)$ = $\int W^{\complement K}_y (x)d\lambda_2(x)$.  In view of \cite[Theorem 7.11.]{smyrnel3}, we have $$\int W^{\complement K}_y (x)d\lambda_i(x) = \int w_y (x)dB^{\complement K}_{i,1}(x) + \int p^{2}_y (x)dB^{\complement K}_{i,2} (x),$$ where $\Lambda: = (\lambda, 0)$, $B_i: = (\lambda_i, 0)$  and  $\Lambda^{\complement K}=(\Lambda^{\complement K}_1, \Lambda^{\complement K}_2)$, $B^{\complement K}_i = (B^{\complement K}_{i,1}, B^{\complement K}_{i,2})$,  are the respective swept pairs  on $\complement K$; therefore $\Lambda^{\complement K}_i$ = $B^{\complement K}_{1,i} - B^{\complement K}_{2,i}$, $(i=1,2)$. According to \cite[Theorem 7.13.]{smyrnel3} (cf. also \cite{smyrnel2}), we have $B^{\complement K}_{1,1} = B^{\complement K}_{2,1}$ , and $\int p^{2}_y (x)dB^{\complement K}_{1,2}(x)$ = $\int p^{2}_y (x)dB^{\complement K}_{2,2}(x)$  or  $P^{*2}_{B^{\complement K}_{1,2}}$ = $P^{*2}_{B^{\complement K}_{2,2}}$  in  $\Omega$; it follows that  $B^{\complement K}_{1,2}$ = $B^{\complement K}_{2,2}$, hence $\Lambda^{\complement K}_1$ = $\Lambda^{\complement K}_2$ = $0$.
\\  $(i) \Rightarrow (ii)$.  The previous arguments can be reversed to prove the converse implication.
\end{proof}

The case of the pair $(0, \mu)$ with $(0,\mu)^{\complement K} = (0,0)$ was studied in \cite{smyrnel2}; it was established that $(0, \mu)^{\complement K}$ = $(0, 0)$  $\Leftrightarrow$ $P^{2*}_\mu$ = $0$  on ${\complement K}$.
\begin{corollary} We suppose that $\mathcal{H}_{1} = \mathcal{H}_{2}$  (that is, $L_1 = L_2$ in the classical case).
Let $(\lambda,0)$ be a pure binormal pair for the compact set $K$. Then, $\lambda$ is $1$- and $2$-normal, while $(\lambda, \lambda)$ is binormal for $K$.
\end{corollary}
\begin{proof} It follows from Theorem \ref{th1} that $w^{*}_{\lambda_1} = w^{*}_{\lambda_2}$; as the pair $(w^{*}_{\lambda}, p^{1*}_{\lambda})$  is adjoint biharmonic
on $\complement K$,  and therefore  compatible,  we  have  $ p^{1*}_{\lambda}=0$  on  $\complement K$,  and by assumption, $p^{1*}_{\lambda} = p^{2*}_{\lambda}$. We also know that $(\lambda,\lambda)^{\complement K} = (\lambda,0)^{\complement K} + (0,\lambda)^{\complement K}$. Consequently, $\lambda$ is $1$- and $2$-normal, while by Definition \ref{def1}, $(\lambda,\lambda)$ is binormal.
\end{proof}

\section{\textbf{SOME EXAMPLES}}\label{sec:sec3}
The functions $u$ such that $\Delta^{2}u=0$ on an open set $U$ of $\R^n$ satisfy a characteristic  mean value property (see \cite{smyrnel5}): $$u(x)=\alpha_x\int ud \mu^{\omega_1}_x(z) -\beta_x\int ud \mu^{\omega_2}_x(z),$$ where  $\omega_i(x_0, R_i)$, $(i = 1, 2)$, are concentric balls with  $0<R_1<R_2$, $\overline \omega_2 \subset U$, \medskip  $\alpha_x$ = $\frac{R^{2}_2 - \rho^{2}}{R^{2}_2 - R^{2}_1}$, $\beta_x$ = $\frac{R^{2}_1 - \rho^{2}}{R^{2}_2 - R^{2}_1}$,   $\rho$ = $\Vert x -x_0\Vert$, $x\in\omega_1$ and $\mu^{\omega_i}_x$, $(i=1,2)$, are the respective harmonic measures.

Let $(w_y, p^{2}_y)$ be the Green's pair in $\R^n$ (cf. \cite{smyrnel7}); it is biharmonic on the open set  $U=\R^n \setminus \{y\}$. If $\overline \omega_2 \subset U$ and  $x\in \omega_1$, then we have $$w_y(x) = \alpha_x \int w_y(z)d\mu^{\omega_1}_x (z)- \beta_x \int w_y(z)d\mu^{\omega_2}_x (z)$$ or $$w^{*}_x(y) = \alpha_x \int w^{*^{}}_z(y)d\mu^{\omega_1}_x (z)-\beta_x \int w^{*}_z(y)d\mu^{\omega_2}_x (z).$$

$\mathbf{1)}$ We consider the compact set $K = \overline \omega_2$; the pair of  measures $(\lambda, 0)$ with $\lambda = \lambda_{1} - \lambda_{2}$,  where  $\lambda_{1} = \epsilon_x + \beta_x \mu^{\omega_2}_x$,  $\lambda_{2} = \alpha_x \mu^{\omega_1}_x$  is a pure binormal pair of measures. We can also take the decomposition  $\lambda = \lambda_1 -\lambda_2$,  where  $\lambda_1 = \alpha_x\epsilon_x + \beta_x\mu^{\omega_2}_x$,  $\lambda_2 = \beta_x\epsilon_x + \alpha_x \mu^{\omega_1}_x$. Moreover, we observe that the pair  $(\lambda, \lambda)$ is a binormal pair for K.

$\mathbf{2})$ Let $\nu$ be a  measure with compact support in $\omega_1$. If $y\in \complement\overline\omega_2$, we obtain
\begin{align*}
\int w^{*}_{\nu}(y)& = \int w^{*}_z(y) \int\alpha_xd\mu^{\omega_1}_xd\nu(x) - \int w^{*}_z(y) \int\beta_xd\mu^{\omega_2}_x(z)d\nu(x)\\
&= \int w^{*}_z(y)d\sigma(z) - \int w^{*}_zd\tau(z) .\end{align*}
The pair $(\lambda, 0)$, where  $\lambda = \nu + \tau - \sigma$ is pure biharmonic, while the pair  $(\lambda, \lambda)$ is a binormal pair.

\begin{note}Obviously, since every compact set is contained in a ball, we can construct pure binormal (resp. binormal)  pairs from a given  measure.\end{note}

$\mathbf{3})$ Starting from a  measure $\lambda$ supported by a compact set $ E\subset \R^n$, G. Choquet and J. Deny (\cite{choquet-deny}) have constructed another measure $\lambda'$ such that $d\lambda' = U^{\lambda}d\tau$  on
$\Hat{E} = E\cup(\underset{i}{\cup} E_i)$,
where the sets $E_{i}$ are the connected components of $\complement E$ with $\overline E_i$ compact,
$U^{\lambda}$ is the potential generated by $\lambda$, and $d\tau$ is the volume element (and so on for the polyharmonic case). The potential $U^{\lambda'}$ is defined by
\begin{align*}
U^{\lambda'}(x)&=\int G_1(x,y)d\lambda'(y)=\int G_1(x,y)U^{\lambda}(y)d \tau(y)\\
&=\int \int G_1(x,y)G_1(y,z)d \tau(y)d\lambda(z)=\int G_2(x,z)d\lambda(z), \end{align*}
where $G_1$ is the newtonian kernel, and $G_2(x,y) = \int G_1(x,z)G_1(z,y)d\tau(z)$ is the iterated kernel (see \cite{nicolesco}). If $U^{\lambda'}(x)= \int G_2(x,z)d\lambda(z) = 0$, on $\complement \Hat{E}$, then $$\Delta_x\int G_2(x,z)d\lambda(z)=\int G_1(x,z)d\lambda(z)=0 \text{ on } \complement \Hat{E};$$ therefore, the pair $(\lambda,\lambda)$ is binormal.

$\mathbf{4})$ Let $(u^{*}_2, 1)$ be a strictly positive adjoint biharmonic pair, and let $V^{*}_2$ be the associated kernel of the potential part of $u^{*}_2$. If $v^{*}_1$ is a nonnegative adjoint $1$-hyperharmonic function, the adjoint pair $(V^{*}_2v^{*}_1, v^{*}_1)$ is a pure hyperharmonic pair; it will be an adjoint pure potential pair, if $v^{*}_1$ is an adjoint $1$-potential, continuous  with a compact harmonic* support.  Let $\lambda$ be a  measure supported by a compact set $K \subset \Omega$; we have, $p^{1*}_\lambda(x) = \int p^{1*}_z (x)d\lambda(z)$  and  $V^{*}_2 1(y) = \int p^{2*}_x(y)d\xi(x)$, where $\xi$  is the  nonnegative  measure associated  to the adjoint potential $V^{*}_2 1$. Now, let $\lambda'$ be another  measure with density $p^{1*}_\lambda$  relative to $\xi$; we consider the following function
\begin{align*}
q^{*}_2(y)& = \int  p^{2*}_x(y)p^{1*}_\lambda(x)d\xi(x) = \int \Big( \int p^{1*}_z(x)p^{2*}_x(y)d\xi(x)\Big)d\lambda(z) \\
&= \int w^{*}_z(y)d\lambda(z)
= V^{*}_2 p^{1*}_\lambda(y) = w^{*}_\lambda(y). \end{align*}
Therefore, if $V^{*}_2 p^{1*}_\lambda = 0$ on $\complement K$, we also have $p^{1*}_\lambda = 0$. Consequently, the pair $(\lambda, 0)$ is pure binormal for $K$. On the other hand, if $\mu$ is a $2$-normal measure for $K$, then the pair $(\lambda, \mu)$ will be binormal for $K$.

\section{\textbf{SOME MEAN VALUES PROPERTIES OF BIHARMONIC PAIRS}}\label{secc4}

 Let us recall some further results on harmonic and biharmonic spaces (cf. \cite[parts X,XI]{smyrnel3}). In a harmonic space, we consider a potential $P$ on $\Omega$, which is finite, continuous, and strictly superharmonic. Let $\xi$ be its associated nonnegative  measure. We define the Dynkin's operators $L$, and $L'$ relative to $P$, as
\begin{equation}\label{dyn}
L_{P}f(x) = \limsup_{\omega\searrow x}\frac{f(x)-\int fd\rho^{\omega}_{x}}{P(x)-\int Pd\rho^{\omega}_{x}}\, ,\quad L'_{P}f(x) = \liminf_{\omega\searrow x}\frac{f(x)-\int fd\rho^{\omega}_{x}}{P(x)-\int Pd\rho^{\omega}_{x}},
\end{equation}
where $x \in \Omega$, $\omega$ is an open set with $\bar{\omega}$ compact, $f$ is a numerical function on $\Omega$ such that the numerator in \eqref{dyn} is defined, and $\rho^{\omega}_{x}$ is the harmonic measure. We can see that $L_{P}f(x) = L_{p^{\omega}}f(x)$ on the harmonic space $\omega$, where $p^{\omega}=P(x)-\int Pd\rho^{\omega}_{x}$, ($x\in\omega$). Moreover, if $V$ is the kernel associated to $P$, then we have $LV\phi = L'V\phi = \phi$ for $\phi\in C_{b}(\Omega)$. The following inequality  $Lu(x)\geq 0$ (or $ L'u(x)\geq 0$) on an open set $U\subset \Omega$ is also characteristic of hyperharmonic functions on $U$.

Let $L^j$, $L^{j'}$ be the operators in \eqref{dyn} associated to the space $(\Omega,\mathcal H_j)$ ($j=1,2$). We say that the pair $(f_1,f_2)$ of finite and continuous functions in the open set $U\subset\Omega$, is \emph{regular} if $L^1f_1$ and $L^2f_2$ (or $L^{1'}f_1$ and $L^{2'}f_2$) are finite and continuous in $U$.

Next, we define the operators
\begin{equation*}
\Gamma_{1}f(x) = \limsup_{\omega\searrow x}\frac{f(x)-\int fd\mu^{\omega}_{x}}{\int d\nu^{\omega}_{x}}\, ,\quad \Gamma'_{1}f(x) = \liminf_{\omega\searrow x}\frac{f(x)-\int fd\mu^{\omega}_{x}}{\int d\nu^{\omega}_{x}}.
\end{equation*}
 Since on a relatively compact open set there exists a strictly positive biharmonic pair $(v_1, v_2)$, we can assume, without loss of generality, that $v_2 = 1$. The Riesz's decomposition yields $v_1=p_1+h_1$, where $p_1$ is a $1$-potential and $h_1$ is an $1$-harmonic function on $\omega$. We have $L_{p_1} f(x) = \Gamma_1f(x)$, as well as $L'_{p_1} f(x) = \Gamma'_1f(x)$, and the relation
$\Gamma_1w_1 \ge w_2$ (or $\Gamma'_1 w_1 \ge w_2$), at the points where $w_1$ is finite, is a  characteristic property of the hyperharmonic pairs
 $(w_1, w_2)$.
\begin{proposition}\label{proposition1} Let $(\lambda, \mu)$ be a binormal pair of  measures supported by a compact set $K \subset U$, where $U$ is an open subset of $\Omega$, and $(u_1, u_2)$ a biharmonic pair of functions on $U$. Then, $\int u_1d\lambda = 0$, and $\int u_2d\mu = 0$.
\end{proposition}
\begin{proof} First, we know that $\int u_2d\mu = 0$ if $\mu$ is a  $2$-normal  measure relative to a compact set $K \subset U$ (cf. \cite[Proposition 1]{smyrnel2}). Now, we shall prove the other equality. We consider a relatively compact open set  $\omega$, such that $K \subset \omega \subset \overline \omega \subset U$. By \cite[Proposition 1.7.]{smyrnel6}, there exist continuous potential pairs $(p_1, p_2)$ and $ (q_1, q_2)$, which are biharmonic on $\omega$, and such that $(u_1, u_2) +(q_1, q_2) = (p_1, p_2)$. We have the decompositions: $(p_1, p_2) = (p'_1, p_2)+(s_1, 0)$,  as well as $(q_1, q_2) = (q'_1, q_2)+(t_1, 0)$, where $(p'_1, p_2)$, $(q'_1, q_2)$ are pure potential pairs in $\Omega$, biharmonic on $\omega$, while $s_1$ and $t_1$ are $1$-potentials in $\Omega$ (see \cite[Proposition 2.8 and Proposition 2.2]{smyrnel6}); moreover, $s_1$ and $t_1$ are $1$-harmonic on $\omega$, since  $\Gamma_1p_1 = \Gamma_1p'_1 = p_2$, $\Gamma_1q_1 = \Gamma_1q'_1 = q_2$ on $\Omega$, while $\Gamma_1(p_1-p'_1) = 0$, $\Gamma_1(q_1-q'_1) = 0$ on $\omega$, (cf. \cite[Corollary 11.4]{smyrnel3}). Therefore, we have on $\omega$: $u_1 = p_1-q_1 = p'_1-q'_1+h_1$, where $h_1 = s_1-t_1$ is $1$-harmonic on $\omega$. Finally, the nonnegative measures $\zeta$ and $\xi$ associated to the pure pairs $(p'_1, p_2)$ and $(q'_1, q_2)$ (cf. \cite[(3.13)]{smyrnel6}), are supported by $\complement \omega$. As  $\int h_1d\lambda = 0$, (cf. \cite[Proposition 1]{smyrnel2}), we obtain
\begin{align*}
\int u_1d\lambda&=\int h_1d\lambda+\int (p'_1-q'_1)d\lambda=\int\int w_y(x)d\theta(y)d\lambda(x)\\
 &=\int \Big( \int w_y(x)d\lambda(x)\Big)d\theta(y)=0,
\end{align*}
where $\theta=\zeta - \xi$, since $\int w_y(x)d\lambda(x)=0$ holds on $\complement K \supset\complement \omega$.
\end{proof}

Next, we shall study the converse of Proposition \ref{proposition1}.
\begin{proposition}\label{proposition3} Let $U$ be an open subset of $\Omega$ and let $(u_1, u_2)$ be a pair of regular functions satisfying $\int u_1d\lambda_i = 0$,\;$\int u_2d\mu_i = 0$ for a family $(\lambda_i, \mu_i)$ of binormal pairs of  measures relative to compact sets $K_i\subset\omega_i$ with $\lambda_i \not=0$, \:$\mu_i\not=0$, such that $P^{1*}_{\lambda_{i,1}} \geq P^{1*}_{\lambda_{i,2}}$, $P^{2*}_{\mu_{i,1}} \geq P^{2*}_{\mu_{i,2}}$, for all $i\in I$,
the open sets $\omega_i$ forming a basis of $U$; then, the pair $(u_1, u_2)$ is biharmonic on $U$.
\end{proposition}
\begin{proof} Let $\omega$ be an open set with $\overline \omega \subset U$ and $\overline \omega$ compact. There is a strictly positive biharmonic pair $(v_1, v_2)$ on $\omega$ (cf. \cite[Theorem 6.9.]{smyrnel3}); without loss of generality, we may assume that $v_2=1$, and we can replace $U$ by $\omega$. In the associated $1$-harmonic space, the Riesz decomposition implies that $v_1 = p_1 + h_1$; we consider the kernel $V^{\omega}_{1}$ associated to the potential $p_1$ and the associated operators $L_1$,\,$\Gamma_1$ (cf. \cite[parts X, XI]{smyrnel3}). The pair $(V^{\omega}_1u_2, u_2)$ is biharmonic since $L_1V^{\omega}_1u_2 = \Gamma_1V^{\omega}_1u_2 = u_2$, and $u_2$ is a $2$-harmonic function (cf. \cite[Proposition 2]{smyrnel2})\footnote{Analogous notions and results are available in the adjoint case.}. It follows from Proposition \ref{proposition1} that $\int V^{\omega}_1u_2d\lambda_i = 0$ holds for all $\lambda_i$ satisfying the assumptions of Proposition \ref{proposition3}.
At this stage, we consider the function $\phi = V^{\omega}_1u_2 - u_1$ on $\omega$; since the functions $V^{\omega}_1u_2$ and $u_1$ are continuous on $\omega$, $\phi$ will also be continuous on $\omega$. Therefore, we obtain $\int \phi d\lambda_i = 0$. In addition, since $p^{1*}_{\lambda_i} = 0$ on $\complement K_i$ (see the beginning of the proof of Theorem \ref{th1}), $\phi$ is in view of \cite[Proposition 3]{smyrnel2} a $1$-harmonic function, denoted by $r_1$. Therefore, $u_1 = V^{\omega}_1u_2 - r_1$ is the first component of a biharmonic pair on $\omega$, namely, of the pair $( V^{\omega}_1u_2 - r_1, u_2)$. Finally, since the pair $(u_1, u_2)$ is biharmonic on every open set $\omega \subset \overline \omega \subset U$, with $\overline \omega$ compact, it will also be biharmonic on $U$.
\end{proof}

\begin{corollary}
Let $L_j$ ($j=1,2$) be a second order linear elliptic operator with regular coefficients defined on a domain $\Omega\subset\R^n$ ($n\geq 2$). We consider the biharmonic space of the solutions of the system $L_1 u_1=-u_2$, $L_2 u_2=0$ on $\Omega$. We suppose that there exists a positive potential pair; therefore there exists a positive $L_j$-potential ($j=1,2$) \cite[Partie XI]{smyrnel3}, and \cite[Chap VII]{herve}. Then, $$u_1(x)=\alpha_x\int u_1 d\mu_x^{\omega_1}-\beta_x\int u_1 d\mu_x^{\omega_2}$$ holds for every $x \in \Omega$, where $\omega_1$, $\omega_2$ are concentric balls such that $x \in \omega_1\subset \overline{\omega_2}\subset\Omega$ (cf. Section \ref{sec:sec3}). This property is characteristic of biharmonic\footnote{The function $u_1$ is called biharmonic on $\Omega$, if it is the first component of a biharmonic pair on $\Omega$.} functions on $\Omega$. We notice that if $L_1=L_2$, then we can also write $u_2(x)=\alpha_x\int u_2 d\mu_x^{\omega_1}-\beta_x\int u_2 d\mu_x^{\omega_2}$.
\end{corollary}

\section{\textbf{PROPERTIES OF BINORMAL PAIRS OF MEASURES}}\label{secc5}

Let $\lambda, \mu$ be  measures, $\lambda=\lambda_1-\lambda_2$, $\mu=\mu_1-\mu_2$, with $\lambda_i \geq 0$, $\mu_i \geq 0$, $(i=1,2)$, and consider the pairs $\Lambda:=(\lambda, 0)$, $B_i:=(\lambda_i, 0)$, as well as the pair $M:=(0, \mu)$. Therefore, we have $\Lambda^{\complement K}_{i} = B^{\complement K}_{1,i} - B^{\complement K}_{2,i}$ and $M^{\complement K} = (0, \mu)^{\complement K}$ (cf. Section \ref{sec1} and the proof of Theorem \ref{th1}).
\begin{theorem}\label{theorem2} The following are equivalent:
\begin{itemize}
\item[(i)] The pair $\Lambda=(\lambda, 0)$ is pure binormal and the pair $M=(0, \mu)$ is $2$-normal.
\item[(ii)] $\Lambda^{\complement K}_{i} = 0$ and $M^{\complement K}_{i} = 0$, $(i=1,2)$.
\item[(iii)] $\int (p_{1} - q_{1})d\lambda = 0$ and $\int (p_{2} - q_{2})d\mu = 0$, where $(p_{1}, p_{2})$, $(q_{1}, q_{2})$ are potentials pairs in $\Omega$ with support in $\complement K$.
\item[(iv)] The previous potential pairs could be pure potential pairs.
\item[(v)] $\int u_1d\lambda = 0$ and $\int u_2d\mu = 0$ hold for every biharmonic pair of functions $(u_1, u_2)$ on an open set $\omega\supset K$.
\item[(vi)] $\lambda = \xi - \Xi^{\complement K}_{1}$, and $\Xi^{\complement K}_{2}=0$, where $(\xi, 0)^{\complement K} = (\Xi^{\complement K}_{1}, \Xi^{\complement K}_{2})$ with $\xi$ the part of $\lambda$ supported by the set of points of $K$ where $\complement K$ is $1$-thin; $\mu = \tau - T^{\complement K}_{2}$, where $(0, \tau)^{\complement K} = (T^{\complement K}_{1}, T^{\complement K}_{2})$, with $\tau$ the part of $\mu$ supported by the set of points where $\complement K$ is $2$-thin.
\end{itemize}
\end{theorem}
\begin{proof}(i) $\Leftrightarrow$ (ii). We have already established the first part of Theorem \ref{theorem2} in the proof of Theorem \ref{th1}. Concerning the second part, we can see that these implications are well known in harmonic spaces (cf. \cite{smyrnel2}).
\\(i) $\Rightarrow$ (v). This is proved in Proposition \ref{proposition1}.
\\(v) $\Rightarrow$ (i). Suppose there exist points $y_1, y_2\in \complement K$ where $w^{*}_{\lambda_1}(y_1)\not=w^{*}_{\lambda_2}(y_1)$,\quad$p^{2*}_{\mu_1}(y_2)\not=p^{2*}_{\mu_2}(y_2)$; we take as $(u_1, u_2)$ the Green's pair $(w_{y}, p^{2}_{y})$ and we have $\int w_{y_1}(x)d\lambda_{1}(x) = \int w_{y_1}(x)d\lambda_{2}(x)$ as well as $\int p^{2}_{y_2}(x)d\mu_{1}(x) = \int p^{2}_{y_2}(x)d\mu_{2}(x)$, therefore we get $w^{*}_{\lambda_1}(y_{1}) = w^{*}_{\lambda_2}(y_{1})$ and $p^{2*}_{\mu_1}(y_{2}) = p^{2*}_{\mu_2}(y_{2})$, which contradict our assumptions (cf. Theorem \ref{th1}).
\\(iii) $\Rightarrow$ (v). By \cite[Proposition 1.7.]{smyrnel6}, there are two continuous potential pairs $(p_1, p_2)$, and $(q_1, q_2)$, which are biharmonic on a relatively compact open set $\omega'$, with $K\subset\omega'\subset\overline{\omega'}\subset\omega$, and such that $u_i = p_i - q_i$ on $\omega'$, (i=1,2).
\\(v) $\Rightarrow$ (iii). We choose an open set $U\supset K$ such that the supports of the potential pairs $(p_1, p_2)$, and $(q_1, q_2)$ are not contained in $U$; hence, these pairs are biharmonic on $U$.
\\(iii) $\Rightarrow$ (iv). This is straightforward because (iv) is a particular case of (iii).
\\(iv) $\Rightarrow$ (i). Suppose there exist two points $y_1, y_2\in \complement K$ such that $w^{*}_{\lambda_1}(y_1)\not=w^{*}_{\lambda_2}(y_1)$ and $p^{2*}_{\mu_1}(y_2)\not=p^{2*}_{\mu_2}(y_2)$. We take as pure potential pairs supported on $\complement K$, the Green's pairs $(w_{y}, p^{2}_{y})$ and $(kw_{y}, kp^{2}_{y})$, where $k>0$, $k\not=1$. Therefore, we obtain $\int (kw_{y_1}(x)- w_{y_1}(x))d\lambda_{1}(x) = \int (kw_{y_1}(x)-w_{y_1}(x))d\lambda_{2}(x)$ and $(k-1)p^{2*}_{\mu_1}(y_2)=(k-1)p^{2*}_{\mu_2}(y_2)$; these conclusions contradict our assumptions (see also Theorem \ref{th1}).
\\ (ii)$\Rightarrow$(vi). $(\lambda, 0)=(\xi, 0)+(\sigma, 0)$, with  $\sigma$  the part of $\lambda$ supported by the set of points where $\complement K$ is not $1$-thin. Setting $\Sigma:=(\sigma, 0)$, we have $(\lambda, 0)^{\complement K}=(\Lambda^{\complement K}_{1}, \Lambda^{\complement K}_{2})=(\Xi^{\complement K}_{1}, \Xi^{\complement K}_{2})+(\Sigma^{\complement K}_{1}, \Sigma^{\complement K}_{2})$. On the other hand, we know that $\Sigma^{\complement K}_{1}=\sigma$. As $\Lambda^{\complement K}_{1}=0$, we deduce that $\Sigma^{\complement K}_{1}+\Xi^{\complement K}_{1}=0$; it follows that $\lambda=\xi+\sigma=\xi-\Xi^{\complement K}_{1}$. Furthermore, since $\Lambda^{\complement K}_{2}=0$, we obtain $\Xi^{\complement K}_{2}+\Sigma^{\complement K}_{2}=0$. Finally, in view of \cite[Remark 2.12]{smyrnel4}), we conclude that $\Sigma^{\complement K}_{2}=0$ (see also \cite[Theorem 7.13.]{smyrnel3}).
\\(vi)$\Rightarrow$(i). We know that $\int P^{\complement K}_{1}d\xi = \int p_1d\Xi^{\complement K}_{1} + \int p_2d\Xi^{\complement K}_{2}$, where $(p_1, p_2)$ is a potential pair; since $\Xi^{\complement K}_{2}=0$, it follows that $\int P^{\complement K}_{1}d\xi = \int p_1d\Xi^{\complement K}_{1}$. Now, if $(p_1, p_2)$ is the Green's pair $(w_{y},p^{2}_{y})$, then $\int W^{\complement K}_{y}(x)d\xi(x) = \int w_{y}(z)d\Xi^{\complement K}_{1}(z)$, that is, $\int W^{*\complement K}_{x}(y)d\xi(x) = \int w^{*}_{z}(y)d\Xi^{\complement K}_{1}(z)$ in view of Lemma \ref{lemma1}. As for $x\in K$, $W^{*\complement K}_{x} = w^{*}_{x}$ holds on $\complement K$, we deduce that $\int w^{*}_{x}(y)d\xi(x) = \int w^{*}_{z}(y)d\Xi^{\complement K}_{1}(z)$; therefore, $w^{*}_{\lambda}=0$ on $\complement K$.

\begin{note} We point out that the implication (vi)$\Rightarrow$(ii) can be established, by reversing the arguments in the proof of (ii)$\Rightarrow$(vi). We can see in the proof of \cite[Proposition 3]{smyrnel2} that $\sigma = -\Xi^{\complement K}_{1}$ holds for every $1$-normal measure.\end{note}

\end{proof}

\begin{theorem}\label{theorem3} Let $K$ be a compact subset of $\Omega$. The following are equivalent:

\begin{itemize}\item[(i)] There exists a pure binormal pair of measures $(\lambda, 0)$ for the compact set $K$, with $\lambda\not=0$.\item[(ii)] $\complement K$ is $1$-thin for at least one point of $K$.\end{itemize}
\end{theorem}
\begin{proof} (i)$\Rightarrow$(ii). In view of Theorem \ref{theorem2}, we have $\Lambda^{\complement K}_{i}=0$, ($i=1,2$), and by assumption, $\lambda\not=0$. If $\complement K$ is not $1$-thin at any point of $K$, then we will obtain
$\lambda=\Lambda^{\complement K}_1$ (cf. \cite[Proposition 3]{smyrnel2});
%$\lambda_{1} = B^{\complement K}_{1,1}$, $\lambda_{2}=B^{\complement K}_{2,1}$, and $B^{\complement K}_{1,2} = B^{\complement K}_{2,2} = 0$;
since $\Lambda^{\complement K}_{1} = B^{\complement K}_{1,1}-B^{\complement K}_{2,1} = 0$, it follows that $\lambda = 0$. This is a contradiction. \\(ii)$\Rightarrow$(i). Given a pure binormal pair $(\lambda, 0)$, suppose that $\lambda = 0$. By assumption and in view of Theorem \ref{theorem2}, we will obtain $\lambda - \Lambda^{\complement K}_{1} = \xi - \Xi^{\complement K}_1 = 0$ and $\xi\not=0$ (since $\complement K$ is $1$-thin for at least one point of $K$). As the measure $\xi$ is supported by the set of unstable points of $K$, and $\Xi^{\complement K}_{1}$ is supported by the set of points where $\complement K$ is not $1$-thin, we deduce that $\xi\not=\Xi^{\complement K}_{1}$ (see \cite[4.6. Proposition]{hansen} and \cite[Lemma VIII,2]{brelot1b}); therefore, $\lambda\not=0$, which is a contradiction.
\end{proof}

We denote by $\mathcal{M}$ the set of measures on $\Omega$. We endow it with the vague topology, that is, the topology of the  simple convergence on the space of continuous functions with compact support. Similarly, we consider the set $\mathcal{M} \times  \mathcal{M}$ with the respective vague topology. We also denote by $\mathcal{K}_{i}$ the set of points of $K$, where $\complement K$ is $i$-thin, and by $\mathcal{N}$ (resp. $\mathcal{N}_{i})$, the set of pure binormal pairs of measures (resp. the set of $i$-normal measures, $i=1,2$) for $K$. Finally, we recall that $(\epsilon_x, 0)^{\complement K} = (\mu^{\complement K}_{x},\nu^{\complement K}_{x})$, where $\mu^{\complement K}_{x}$ is the swept measure of $\epsilon_x$ on $\complement K$ in the $1$-harmonic space, and $(0, \epsilon_x)^{\complement K} = (0, \lambda^{\complement K}_{x})$, where $\lambda^{\complement K}_{x}$ is the swept measure of $\epsilon_x$ on $\complement K$ in the $2$-harmonic space.

\begin{theorem}\label{theorem6}

\

\begin{itemize}
\item[(i)] The pairs $(\epsilon_x-\mu^{\complement K}_{x}, \nu^{\complement K}_{x})$, where $x \in \mathcal{K}_1$, form a total subset of $\mathcal{N}$.
\item[(ii)] The measures $(\epsilon_x-\mu^{\complement K}_{x})$, where $x \in \mathcal{K}_1$, form a total subset of $\mathcal{N}_1$.
\item[(iii)] The measures $(\epsilon_x-\lambda^{\complement K}_{x})$, where $x \in \mathcal{K}_2$, form a total subset of $\mathcal{N}_2$.
\end{itemize}
\end{theorem}
\begin{proof} (i) First, it is well known that the swept pair $(\Xi^{\complement K}_{1},\Xi^{\complement K}_{2})$ of $(\xi, 0)$ on $\complement K$, is expressed by $\Xi^{\complement K}_{1}(f)$= $\int \mu^{\complement K}_{x}(f)d\xi(x)$, $\Xi^{\complement K}_{2}(f)$ =$ \int \nu^{\complement K}_{x}(f)d\xi(x)$, where $f$ is any continuous function with compact support. Next, we recall that by definition of the integral, there exist points $x_{n}$ of $\mathcal{K}_1$ such that
\begin{equation}\label{eq1}
\mid\int \mu^{\complement K}_{x}(f)d\xi(x)-\overset{N}{\underset{n=1}{\sum}}\lambda_n\mu^{\complement K}_{x_n}(f) \mid<\epsilon' \text{ with } \overset{N}{\underset{n=1}{\sum}}\lambda_n = \xi(\mathcal{K}_1).\footnote{By considering a suitable partition of $\mathcal{K}_1$, we can choose the (same) coefficients $\lambda_j$, such that relations \eqref{eq1} and \eqref{eq2} are satisfied (cf. \cite[p. 108-109]{chilov}, \cite{bourbaki}, and \cite[p. 126-127]{cour}.}
\end{equation}
Moreover, according to \cite[Theorem 1, chap.III, \S 2, No. 4]{bourbaki}, there exists a linear combination $ \overset{p}{\underset{j=1}{\sum}}\lambda_j\epsilon_{x_j}$ such that
\begin{equation}\label{eq2}
\mid \overset{p}{\underset{j=1}{\sum}}\lambda_j\epsilon_{x_j}(f)-\xi(f)\mid<\epsilon'' \text{ and } \overset{p}{\underset{j=1}{\sum}}\lambda_j = \xi(\mathcal{K}_1).
\end{equation}
Consequently, combining \eqref{eq1} and \eqref{eq2}, we can write
$$\overset{q}{\underset{i=1}{\sum}}\lambda_i(\epsilon_{x_i}-\mu^{\complement K}_{x_i})(f)-\epsilon\leq \xi(f)-\Xi^{\complement K}_{1}(f)\leq \overset{q}{\underset{i=1}{\sum}}\lambda_i(\epsilon_{x_i}-\mu^{\complement K}_{x_i})(f) + \epsilon ,$$ $$\text{and }-\epsilon \leq  \overset{q}{\underset{i=1}{\sum}}\lambda_i\nu^{\complement K}_{x_i}(f)\leq \epsilon\quad with\quad \overset{q}{\underset{i=1}{\sum}}\l_i = \xi(\mathcal{K}_1). $$
Since, by Theorem \ref{theorem2}, $(\lambda, 0) = (\xi, 0) - (\Xi^{\complement K}_{1}, \Xi^{\complement K}_{2})$, the result follows. Assertions (ii) and (iii) can be proved in the same way.
\end{proof}

Finally, we shall examine how normal and binormal measures are connected. \begin{proposition}\label{proposition5}

\

\begin{itemize}
\item[(i)] If $(\lambda, 0)$ is a pure binormal pair for the compact set $K$, then the measure $\lambda$ is $1$-normal for $K$.\item[(ii)] Conversely, suppose that $\lambda$ is a $1$-normal measure for $K$. Then the pair $(\lambda, 0)$ is not necessarily a pure binormal pair, even if the $j$-harmonic spaces coincide ($j=1,2$).
\end{itemize}
\end{proposition}
\begin{proof} (i). Since the pair $(w^{*}_{\lambda}, p^{1*}_{\lambda})$ is biharmonic adjoint on $\complement K$, and therefore compatible, then $w^{*}_{\lambda_1} = w^{*}_{\lambda_2}$ on $\complement K$ implies that $p^{1*}_{\lambda_1} = p^{1*}_{\lambda_2}$ there. Consequently, $\lambda$ is $1$-normal for $K$.\\(ii). If $p^{1*}_{\lambda_1}=p^{1*}_{\lambda_2}$ on $\complement K$, then we assert that $w^{*}_{\lambda_1}=w^{*}_{\lambda_2}+h^{*}_{2}$ on $\complement K$, where $h^{*}_{2}$ is  an adjoint $2$-harmonic function on $\complement K$. Indeed, since the pure potential pairs satisfy  the relations $\Gamma^{*}_{1}w^{*}_{\lambda_1} = p^{1*}_{\lambda_1}$ and $\Gamma^{*}_{1}w^{*}_{\lambda_2}=p^{1*}_{\lambda_2}$ on $\complement K$, we obtain $\Gamma^{*}_{1}(w^{*}_{\lambda_1} - w^{*}_{\lambda_2}) = 0$ on $\complement K$, that is, $w^{*}_{\lambda_1} - w^{*}_{\lambda_2} = h^{*}_{2}$, where\:$h^{*}_{2}$  is an adjoint $2$-harmonic function on the complement of $K$.

Let us now take for the elliptic operator, the Laplacian. We have the inclusion $$\{\lambda: (\lambda, 0)\text{ is pure binormal}\}\subset \{\lambda: \lambda \text{ is a $1$-normal measure}\}.$$ For instance, the measure $\lambda = \epsilon_x - \mu^{\omega}_{x}$, where $\omega$ is a ball, is normal for $K = \overline{\omega}$, but the pair $(\lambda, 0)$ is not a pure binormal pair. On the other hand, the measure $\lambda = \epsilon_x - \alpha_x \mu^{\omega_1}_{x} + \beta_x \mu^{\omega_2}$ is $1$-normal, and the pair $(\lambda, 0)$ is also pure binormal ($\omega_1$, $\omega_2$  are concentric balls, and $\omega_1 \subset \overline{\omega_1}\subset \omega_2$).
\end{proof}
Nevertheless, in general we have:
\begin{proposition}\label{proposition7} Let $\lambda$ be a $1$-normal measure for $K$. Then, $(\lambda, 0)$ is a pure binormal pair iff $\Xi^{\complement K}_{2} = 0$.
\end{proposition}
\begin{proof} This follows immediately from Theorem \ref{theorem2}.
\end{proof}

\nocite{*}
\bibliographystyle{plain}

\hspace{70mm}\scriptsize \emph{Emmanuel P. Smyrnelis}

\hspace{70mm}\scriptsize \emph{University of Athens}

\hspace{70mm}\scriptsize \emph{Department of Mathematics}

\hspace{70mm}\scriptsize \emph{15784 Athens, Greece}

\hspace{70mm}\scriptsize \emph{smyrnel@math.uoa.gr}

\

\hspace{70mm}\scriptsize \emph{Panayotis Smyrnelis}

\hspace{70mm}\scriptsize \emph{University of Athens}

\hspace{70mm}\scriptsize \emph{Department of Mathematics}

\hspace{70mm}\scriptsize \emph{15784 Athens, Greece}

\hspace{70mm}\scriptsize \emph{smpanos@math.uoa.gr}
\end{document}